\documentclass[reqno,12pt,a4paper]{amsart}

\usepackage[utf8]{inputenc}

\usepackage{enumerate}
\usepackage{amstext,amsmath,amsthm,amsfonts,amssymb,amscd}
\usepackage{latexsym,mathrsfs,dsfont,euscript}
\usepackage{fullpage}
\usepackage{xcolor}

\usepackage{hyperref} 
\hypersetup{
  colorlinks,%
  citecolor=blue,%
    filecolor=red,%
    linkcolor=red,%
    urlcolor=blue
}

\newtheorem{theorem}{Theorem}[section]
\newtheorem*{thm}{Theorem}%[section]
\newtheorem{proposition}[theorem]{Proposition}
\newtheorem{lemma}[theorem]{Lemma}

\theoremstyle{definition}

\theoremstyle{remark}

\numberwithin{equation}{section}

\newcommand{\abs}[1]{\left\vert#1\right\vert}

\newcommand{\proin}[2]{\left<#1,#2\right>}
\newcommand{\norm}[1]{\left\Vert#1\right\Vert}
\newcommand{\normiii}[1]{{\left\vert\kern-0.25ex\left\vert\kern-0.25ex\left\vert #1
		\right\vert\kern-0.25ex\right\vert\kern-0.25ex\right\vert}}

\allowdisplaybreaks

\begin{document}
\title[]{Boundedness and concentration of random singular integrals defined by wavelet summability kernels}

%%    Information for authors

\author[]{Hugo Aimar}
%\email{haimar@santafe-conicet.gov.ar}
%%%
%%%
\author[]{Ivana G\'{o}mez}
%\email{ivanagomez@santafe-conicet.gov.ar}
%%
%%
\thanks{This work was supported by the MINCYT in Argentina: CONICET and Agencia I+D+i; and UNL}

%%   Information for Manuscript


\begin{abstract}
We use Cram\'er-Chernoff type estimates in order to study the Calder\'on-Zygmund structure of the kernels $\sum_{I\in\mathcal{D}}a_I(\omega)\psi_I(x)\psi_I(y)$ where $a_I$ are subgaussian independent random variables and $\{\psi_I: I\in\mathcal{D}\}$ is a wavelet basis where $\mathcal{D}$ are the dyadic intervals in $\mathbb{R}$. We consider both, the cases of standard smooth wavelets and the case of the Haar wavelet.
\end{abstract}
\maketitle

\section{Introduction}

Set $\mathcal{D}$ to denote the family of all dyadic intervals in $\mathbb{R}$. Then $\mathcal{D}=\cup_{j\in\mathbb{Z}}\mathcal{D}^j$ with $\mathcal{D}^j=\{I^j_k=[k2^{-j},(k+1)2^{-j}):k\in\mathbb{Z}]\}$ the sequence of all dyadic intervals with length $2^{-j}$. We shall use the notation $\psi_I(x)=\psi_{j,k}(x)=2^{j/2}\psi(2^jx-k)$ with $I=I^j_k\in\mathcal{D}$, to denote the orthonormal wavelet basis of $L^2(\mathbb{R})$ generated by the wavelet function $\psi$. 

The basic kernels associated to the unconditional character of $\{\psi_I:I\in\mathcal{D}\}$ as a basis for $L^p(\mathbb{R})$ with $1<p<\infty$, are given by series of the form
\begin{equation}\label{eq:kernelTypeOmega}
K(x,y) = \sum_{I\in\mathcal{D}} \omega_I \psi_I(x)\psi_I(y)
\end{equation}
with $\omega_I=\pm 1$ for each $I\in\mathcal{D}$.

Under some mild conditions on $\psi$ the kernels $K(x,y)$ become Calder\'on--Zygmund type kernels and the induced operators are bounded in $L^p(\mathbb{R})$ for $1<p<\infty$.

The standard use of the Calderón-Zygmund theory in the proof of the unconditionality of wavelet bases in Lebesgue and Sobolev spaces (see \cite{MeyerBook90}), is based on the estimates
\begin{equation*}
\sum_{I\in\mathcal{D}}\abs{\psi_I(x)}\abs{\psi_I(y)} \leq \frac{C}{\abs{x-y}}
\end{equation*}
and
\begin{equation*}
\sum_{I\in\mathcal{D}}\abs{\frac{d\psi_I}{dx}(x)}\abs{\psi_I(y)} \leq \frac{C}{\abs{x-y}^2}.
\end{equation*}
See Chapter~9 in \cite{Daubechies92}. Nevertheless, when $\abs{\psi(x)}+\abs{\psi'(x)}\leq\frac{C}{(\abs{x}+1)^{1+\varepsilon}}$, these estimates only work for the original kernel $K_\sigma(x,y)=\sum_{I\in\mathcal{D}}\sigma_I \psi_I(x)\psi_I(y)$, with $\abs{\sigma_I}=1$ or for bounded sequences $\sigma_I$. Actually, in the application to the proof of unconditionally, $\sigma_I$ is the Rademacher sequence of independent identically distributed random variables which take only the values $+1$ and $-1$.

On the other hand, a simple classical case which is not covered by this approach is the case of the Haar  wavelet. The size estimate of the kernel holds, nevertheless there is not enough  regularity. As shown in \cite{AiGoPetermichl18}, we recover size and regularity  estimates of Calder\'on-Zygmund type, after changing the underlying metric. The right metric is the dyadic distance, instead of the Euclidean one.

For a sequence of independent, unbounded random variables, $a_I(\omega)$, $I\in\mathcal{D}$, defined on a probability space $(\Omega,\mathscr{P})$, the kernel
\begin{equation*}
K(x,y;\omega)=\sum_{I\in\mathcal{D}} a_I(\omega)\psi_I(x) \psi_I(y),
\end{equation*}
with $\psi$ a good wavelet as in Chapter~9 of \cite{Daubechies92}, is not even well defined. We aim to use Cram\'er-Chernoff method in order to prove that the kernels $K(x,y;\omega)$ are Calder\'on-Zygmund kernels valued on $L^2(\Omega,\mathscr{P})$ when $a_I$ are independent subgaussian random variables with variance factors bounded above.

The main results of this paper are the following. First we prove the almost sure convergence of the series 
$\sum_{I\in\mathcal{D}} a_I(\omega)\psi_I(x) \psi_I(y)$ for $x\neq y$, $a_I$ independent and uniformly subgaussian. Second, we prove that for $a_I$ independent and uniformly subgaussian and for $\abs{\psi(x)}+\abs{\psi'(x)}\leq\frac{C}{(1+\abs{x})^{1+\varepsilon}}$, $\varepsilon>0$, the operator $T: f\to \sum_{I\in\mathcal{D}} a_I(\omega)\proin{f}{\psi_I}\psi_I$ is bounded from $L^2(\mathbb{R})$ to $L^2(L^2(\Omega,d\mathscr{P}); dx)$. Third, we show that under the same assumptions, $K(x,y;\cdot)$ is a Calder\'on-Zygmund kernel valued in $L^2(\Omega, d\mathscr{P})$ and hence that $T$ is bounded from $L^p(\mathbb{R})$ to $L^p(L^2(\Omega,d\mathscr{P});dx)$. Then, we extend the above to the Haar case, with the size and smoothness estimates for the kernel provided by the dyadic distance $\delta(x,y)$ in $\mathbb{R}^+$ instead of $\abs{x-y}$. As a byproduct we prove concentration type inequalities for the random kernels about their mean value kernels, and for the random operators about the operator induced by these mean value kernel.

Section~\ref{sec:CramerChernoffMethods} is devoted to introduce the basic result regarding Cram\'er-Chernoff method and  subgaussian random variables. We also review in this section a classical theorem due to Kolmogorov, the so called ``Three Series Theorem'' that we shall use in the prove of the almost sure convergence of the series defining the kernels.

In Section~\ref{sec:AlmostEverywhereConvergence} we deal with the problem of convergence of the series for almost every $\omega$ when the $a_I(\omega)$ are independent and uniformly subgaussian.

In Section~\ref{sec:KvaluedCZPsiSmooth} we prove the $L^2$ boundedness of $T$ and the Calder\'on-Zygmund estimates of the kernel with respect to the norm of $L^2(\Omega,d\mathscr{P})$. Section~\ref{sec:KvaluedCZHaarwavelet} is devoted to the case of the Haar system. Finally in Section~\ref{sec:Concentration} we consider the concentration inequalities.

\section{The Cram\'er-Chernoff bounding method and subgaussian random variables}\label{sec:CramerChernoffMethods}

For the sake of completeness, we shall briefly review in this section our main tool. Namely the Cram\'er-Chernoff method (\cite{Cramer38}, \cite{Chernoff52}). In doing  so we shall follow the lines of \cite{ConcentrationBook}. The starting point is Markov's inequality for the distribution of a random variable with finite expected value.

Let $(\Omega,\mathscr{F},\mathscr{P})$ be a probability space and let $X$ be a random variable with $\mathscr{E}\abs{X}<\infty$. In other words $X\in L^1(\Omega, d\mathscr{P})$. In the search of estimates for the tail probabilities of $X$ about its mean $\mathscr{E}X$, we have to consider, for $t>0$, the two probabilities
\begin{equation*}
\mathscr{P}\{X-\mathscr{E}X\geq t\} \textrm{ and } \mathscr{P}\{\mathscr{E}X-X\geq t\}.
\end{equation*}
Since for $\lambda>0$ fixed, the function of $t>0$ given by $e^{\lambda t}$ is increasing, from Markov's inequality we obtain
\begin{equation*}
\mathscr{P}\{X-\mathscr{E}X\geq t\} \leq e^{-\lambda t}\mathscr{E}e^{\lambda (X-\mathscr{E}X)}
\end{equation*}
and
\begin{equation*}
\mathscr{P}\{\mathscr{E}X - X\geq t\} \leq e^{-\lambda t}\mathscr{E}e^{\lambda (\mathscr{E}X - X)}.
\end{equation*}
The logarithmic moment-generating function
\begin{equation*}
\eta_{X-\mathscr{E}X}(\lambda) =\log \mathscr{E}e^{\lambda (X-\mathscr{E}X)}
\end{equation*}
plays an important role in Cram\'er-Chernoff argument and provide an easy way to generalize normality. Following \cite{ConcentrationBook} we say that an integrable random variable $X$ belongs to $\mathscr{G}(\nu)$ or that $X$ is subgaussian with variance factor $\nu>0$ if the inequality $\eta_{X-\mathscr{E}X}(\lambda)\leq\lambda^2\frac{\nu}{2}$ holds for every $\lambda\in\mathbb{R}$.
\begin{proposition}\label{propo:inequalitiesProbalityandExponential}
	If $X\in\mathscr{G}(\nu)$, then
	\begin{equation*}
	\mathscr{P}\{X-\mathscr{E}X\geq t\} \leq e^{-\frac{t^2}{2\nu}}
	\end{equation*}
	and
	\begin{equation*}
	\mathscr{P}\{\mathscr{E}X - X\geq t\} \leq e^{-\frac{t^2}{2\nu}},
	\end{equation*}
	for every $t>0$.
\end{proposition}
\begin{proof}
	Let us consider the first estimate. Since
	\begin{equation*}
	\mathscr{P}\{X-\mathscr{E}X\geq t\} \leq e^{-\lambda t}\mathscr{E}e^{\lambda (X-\mathscr{E}X)}
	\end{equation*}
	for every $\lambda\geq 0$, then
	\begin{equation*}
	\log\mathscr{P}\{X-\mathscr{E}X\geq t\} \leq -\lambda t + \eta_{X-\mathscr{E}X)}(\lambda)\leq -\lambda t + \lambda^2\frac{\nu}{2},
	\end{equation*}
	for every $\lambda\geq 0$. Hence
	\begin{equation*}
	\log\mathscr{P}\{X-\mathscr{E}X\geq t\} \leq \inf_{\lambda\geq 0}\left(\lambda^2\frac{\nu}{2} -\lambda t\right) = -\frac{t^2}{2\nu}
	\end{equation*}
	and we are done.
\end{proof}

Notice that every normally distributed random variable is subgaussian. Observe also that the Rademacher random variables are all subgaussian with $\nu=1$. This fact follows from Hoeffding's Lemma (\cite{Hoeffding63}) that shows that every bounded random variable is subgaussian. But, of course, not every subgaussian random variable is bounded since normal random variables are subgaussian.

\begin{proposition}\label{propo:sumRVinGofsum}
	Assume that $X_1, \ldots, X_n$ are independent random variables with $\mathscr{E}\abs{X_j}<\infty$ for every $j=1,\ldots,n$ and that $X_j\in\mathscr{G}(\nu_j)$. Then $S=\sum_{j=1}^n X_j$ belongs to $\mathscr{G}\left(\sum_{j=1}^n \nu_j\right)$.
\end{proposition}
\begin{proof}
	Since $\mathscr{E}S= \sum_{j=1}^n \mathscr{E}X_j$, from independence we have that
	\begin{align*}
	\eta_{S-\mathscr{E}S}(\lambda) &=\log\mathscr{E} e^{\lambda (S-\mathscr{E}S)}\\
	&= \log\mathscr{E} e^{\lambda\sum_{j=1}^n(X_j -\mathscr{E}X_j)}\\
	&= \log\mathscr{E}\prod_{j=1}^n e^{\lambda (X_j -\mathscr{E}X_j)}\\
	&= \log\prod_{j=1}^n\mathscr{E}e^{\lambda (X_j -\mathscr{E}X_j)}\\
	&= \sum_{j=1}^n\log\mathscr{E}e^{\lambda (X_j -\mathscr{E}X_j)}\\
	&= \sum_{j=1}^n\eta_{X_j-\mathscr{E}X_j}(\lambda)\\
	&\leq \frac{\lambda^2}{2}\left(\sum_{j=1}^n \nu_j\right),
	\end{align*}
	as desired.
\end{proof}
The above result extends to series of independent random variables with convergence in the $L^2(\Omega,d\mathscr{P})$ sense, provided that the series $\sum_{j\geq 1}\mathscr{E}\abs{X_j}$ and $\sum_{j\geq 1} \nu_j$ both converge.
\begin{proposition}
	Let $\{X_j: j\geq 1\}$ be a sequence of independent random variables with $\sum_{j\geq 1}\mathscr{E}\abs{X_j}<\infty$, $X_j\in\mathscr{G}(\nu_j)$, and $\sum_{j\geq 1}\nu_j=\nu<\infty$. Then, the series $\sum_{j\geq 1} X_j$ converges in $L^2(\Omega, d\mathscr{P})$ to a random variable $S$. Moreover, $\norm{S - \mathscr{E}S}^2_{L^2(\Omega,d\mathscr{P})}\leq 2\nu$.
\end{proposition}
\begin{proof}
	For $n\geq 1$, set $S_n=\sum_{j=1}^n X_j$. Then, with $n>m\geq 1$ we have that 
	\begin{equation*}
	\norm{S_n - S_m}_{L^2(\Omega,d\mathscr{P})} = \norm{\sum_{j=m+1}^n X_j}_{L^2(\Omega,d\mathscr{P})} \leq \norm{\sum_{j=m+1}^n(X_j-\mathscr{E}X_j)}_{L^2(\Omega,d\mathscr{P})} + \sum_{j=m+1}^n \mathscr{E}\abs{X_j}.
	\end{equation*}
	Since $\sum_{j=m+1}^n \mathscr{E}\abs{X_j}<\infty$, the second term above tends to zero for $m\to\infty$. For the first term we use Propositions~\ref{propo:inequalitiesProbalityandExponential} and \ref{propo:sumRVinGofsum}
	\begin{align*}
	\norm{\sum_{j=m+1}^n(X_j-\mathscr{E}X_j)}_{L^2(\Omega,d\mathscr{P})}^2 
	&= \int_{\Omega} \abs{\sum_{j=m+1}^n(X_j-\mathscr{E}X_j)}^2 d\mathscr{P}\\
	&= \int_{\Omega}\left(\int_0^{\abs{\sum_{j=m+1}^n(X_j-\mathscr{E}X_j)}^2} dt\right) d\mathscr{P}\\
	&= \int_0^\infty \mathscr{P} \left\{\abs{\sum_{j=m+1}^n(X_j-\mathscr{E}X_j)} >\sqrt{t}\right\} dt\\
	&\leq \int_0^\infty e^{-\tfrac{t}{2\sum_{j=m+1}^n \nu_j}} dt\\
	&= 2\sum_{j=m+1}^n \nu_j,
	\end{align*}
	which tends to zero for  $m$ tending to infinity.
	Notice that
	\begin{equation*}
	\norm{S-\mathscr{E}S}_{L^2(\Omega,d\mathscr{P})} = \lim_{n\to\infty}\norm{S_n-\mathscr{E}S_n}_{L^2(\Omega,d\mathscr{P})}\leq \limsup_{n\to\infty} 2\sum_{j=1}^n \nu_j = 2\nu.
	\end{equation*}
\end{proof}
The above proposition extends to higher order moments of $S-\mathscr{E}S$. This fact together with Theorem~2.1 on page 25 in \cite{ConcentrationBook} will allow to show that $S$ is also a subgaussian random variable. Theorem~2.1 in \cite{ConcentrationBook} proves that for a random variable $X$ with $\mathscr{E}X=0$, we have that $X\in\mathscr{G}(4C)$ provided that $\mathscr{E}X^{2k}\leq k! C^k$ for every $k=1,2,3,\ldots$

\begin{proposition}\label{propo:HigherOrderforS}
	Let $\{X_j: j\geq 1\}$ be a sequence of independent random variables with $\sum_{j\geq 1}\mathscr{E}\abs{X_j}<\infty$, $X_j\in\mathscr{G}(\nu_j)$ and $\sum_{j\geq 1}\nu_j=\nu<\infty$. Then, the series $\sum_{j\geq 1}X_j$ converges in $L^{2k}(\Omega,d\mathscr{P})$ to random variable $S$ for every integer $k\geq 1$ and 
	\begin{equation*}
	\norm{S-\mathscr{E}S}^{2k}_{L^{2k}(\Omega,d\mathscr{P})} \leq k! (2\nu)^k.
	\end{equation*}
\end{proposition}
\begin{proof}
	Set as before $S_n=\sum_{j=1}^n X_j$. Again, the Cauchy character of $S_n$ in $L^{2k}(\Omega,d\mathscr{P})$ is determined by the behavior of the tail norms
	\begin{align*}
	\norm{\sum_{j=m+1}^n (X_j-\mathscr{E}X_j)}^{2k}_{L^{2k}(\Omega,d\mathscr{P})} 
	&= \int_{\Omega}\left(\int_0^{\abs{\sum_{j=m+1}^n (X_j-\mathscr{E}X_j)}^{2k}} dt\right) d\mathscr{P}\\
	&\leq \int_0^\infty\mathscr{P}\left\{\abs{\sum_{j=m+1}^n (X_j-\mathscr{E}X_j)}> t^{\tfrac{1}{2k}}\right\}  dt\\
	&\leq \int_0^{\infty} e^{-\frac{(t^{1/2k})^2}{2\sum_{m+1}^n\nu_j}} dt\\
	&= \Bigl(2\sum_{m+1}^n\nu_j\Bigr)^k \int_0^\infty e^{-s} s^{k-1} ds\\
	&= k\Gamma(k) \Bigl(2\sum_{m+1}^n\nu_j\Bigr)^k\\
	&= k! \Bigl(2\sum_{m+1}^n\nu_j\Bigr)^k.
	\end{align*}
	This estimate proves both, the convergence of the series in $L^{2k}(\Omega, d\mathscr{P})$ and the inequality
	\begin{equation*}
	\norm{S-\mathscr{E}S}^{2k}_{L^{2k}(\Omega, d\mathscr{P})} \leq k! (2\nu)^k.
	\end{equation*}
\end{proof}

\begin{proposition}\label{propo:ConvergenceSerieRVwithfiniteserieexpectation}
	Let $\{X_j: j\geq 1\}$ be a sequence of independent random variables with $\sum_{j\geq 1} \mathscr{E}\abs{X_j}<\infty$, $X_j\in\mathscr{G}(\nu_j)$ and $\sum_{j\geq 1} \nu_j=\nu<\infty$. Then $S=\sum_{j\geq 1} X_j$ converges in $L^p(\Omega, d\mathscr{P})$ for every $p<\infty$ and $S\in\mathscr{G}(8\nu)$.
\end{proposition}
\begin{proof}
	Follows from Proposition~\ref{propo:HigherOrderforS} and Theorem~2.1 in \cite{ConcentrationBook}.
\end{proof}
For completeness we finish this section with a well known result in Probability Theory that shall be used in the further development of the main results. From \cite{ChungBook} we take the following statement of Kolmogorov's Three Series Theorem regarding the a.e. convergence of series of independent  random variables.

\begin{thm}[Kolmogorov's Theorem (\cite{ChungBook}, Chapter~5)]
	If $(X_n)$ is a sequence of independent random variables and $A$ is a positive number, the almost everywhere convergence of the series $\sum_n X_n$ is equivalent to the simultaneous convergence of the following three numerical series
	\begin{enumerate}[(i)]
		\item $\sum_n \mathscr{P}\{X_n\neq Y_n\}$;
		\item $\sum_n \mathscr{E}(Y_n)$ and
		\item $\sum_n \sigma^2(Y_n)$,
	\end{enumerate}
	with
	\begin{equation*}
	Y_n =
	\begin{cases}
	X_n & if \abs{X_n}\leq A\\
	0 & if \abs{X_n} > A.
	\end{cases}
	\end{equation*}
\end{thm}
For the distribution, mean and variance of truncations we have the following straightforward result.
\begin{lemma}
	Let $X:\Omega\to\mathbb{R}$ be a random variable with finite variance. Let $A$ be a given positive number and
	\begin{equation*}
	X_A(\omega) =
	\begin{cases}
	X(\omega) & if \abs{X(\omega)}\leq A\\
	0 & if \abs{X(\omega)} > A.
	\end{cases}
	\end{equation*}
	Then the distribution measure $\mu_{X_A}$ of $X_A$ is related to the distribution measure $\mu_X$ of $X$ by $\mu_{X_A}(B)=\mu_X((B\cap[-A,A])\setminus\{0\})+ (\mu_X(\{0\})+\mu_X([-A,A]^c))\delta_0(B)$,
	where $\delta_0$ is the Dirac delta at the origin, $[-A,A]$ is the closed interval $-A\leq x\leq A$, $[-A,A]^c=\mathbb{R}\setminus [-A,A]$ and $B$ is a one dimensional Borel set. Hence
	\begin{equation*}
	\mathscr{E}(X_A)=\int_{[-A,A]} xd\mu_X(x),
	\end{equation*}
	and
	\begin{equation*}
	Var(X_A)=\int_{[-A,A]} x^2d\mu_X(x) - \left(\int_{[-A,A]}x\mu_X(x)\right)^2.
	\end{equation*}
\end{lemma}

\section{The almost everywhere convergence of the series $K(x,y;\omega)$}\label{sec:AlmostEverywhereConvergence}
The main result of this section is the almost sure convergence of the series

\begin{equation*}
\sum_{I\in\mathcal{D}} a_I(\omega)\psi_I(x) \psi_I(y)=K(x,y;\omega)
\end{equation*}
for $x\neq y$, $\abs{\psi(x)}\leq\frac{C}{(1+\abs{x})^{1+\varepsilon}}$ and $\{a_I: I\in\mathcal{D}\}$ independent random variables in $\mathscr{G}(\nu)$ for some $\nu>0$.
\begin{theorem}\label{thm:ConvergenceSeriesWaveletpsi}
Let $\psi$ be such that there exist positive constants $C$ and $\varepsilon$ with $\abs{\psi(x)}\leq C(1+\abs{x})^{-1-\varepsilon}$ for every $x\in\mathbb{R}$. Assume that $\widetilde{a}=\{a_I: I\in\mathcal{D}\}$ is a sequence of independent random variables on the probability space $(\Omega,\mathscr{F},\mathscr{P})$ such that $\{a_I: I\in\mathcal{D}\}\subset \mathscr{G}(\nu)$ for some positive $\nu$ and $\sum_I \mathscr{E}\abs{a_I}<\infty$. Then for every $x\neq y$ the series
\begin{equation*}
K_{\widetilde{a}}(x,y;\omega)= \sum_{I\in\mathcal{D}}a_{I}(\omega)\psi_I(x)\psi_I(y)
\end{equation*}
converges for almost every $\omega\in\Omega$.
\end{theorem}
\begin{proof}
	Fix $x\neq y$ both in $\mathbb{R}$. We shall use the Three Series Theorem of Kolmogorov in order to prove the desired convergence. Notice first that since $\sum_I \mathscr{E}\abs{a_I}$ converges it is enough to prove the convergence of the series $\sum_I (a_I(\omega)-\mathscr{E}a_I)\psi_I(x)\psi_I(y)$ for almost every $\omega\in\Omega$. Set $X_I(x,y)(\omega)=(a_I(\omega)-\mathscr{E}a_I)\psi_I(x)\psi_I(y)$. Take $A>0$ fixed. Define
	\begin{equation*}
	Y_I (x,y)(\omega) =
	\begin{cases}
	X_I (x,y)(\omega) & if \abs{X_I (x,y)(\omega)}\leq A\\
	0 & if \abs{X_I (x,y)(\omega)} > A.
	\end{cases}
	\end{equation*}
	Let us start by checking (i) in Kolmogorov's Theorem. For fixed $I\in\mathcal{D}$ we have
	\begin{align*}
	\mathscr{P}\{X_I(x,y)\neq Y_I(x,y)\} &= \mathscr{P}\{\abs{X_I(x,y)}>A\}\\
	&= \mathscr{P}\{\abs{a_I-\mathscr{E}a_I}\abs{\psi_I(x)}\abs{\psi_I(y)}>A\}\\
	&= \begin{cases}
	0 & if \abs{\psi_I(x)}\abs{\psi_I(y)}=0,\\
	\mathscr{P}\{\abs{a_I-\mathscr{E}a_I} > \frac{A}{\abs{\psi_I(x)}\abs{\psi_I(y)}} \} & if \abs{\psi_I(x)}\abs{\psi_I(y)}\neq 0
	\end{cases}\\
	&\leq e^{-\tfrac{1}{2\nu}\tfrac{A^2}{\abs{\psi_I(x)}^2\abs{\psi_I(y)}^2}},
	\end{align*}
	the last inequality follows from Proposition~\ref{propo:inequalitiesProbalityandExponential} since the random variables $a_I$ are uniformly in $\mathscr{G}(\nu)$. Now, since the estimates in \cite{Daubechies92}, we have that $\sum_{I\in\mathcal{D}}\abs{\psi_I(x)}\abs{\psi_I(y)}$ converges for $x\neq y$, so does the series
	\begin{equation*}
	\sum_{I\in\mathcal{D}}\mathscr{P}\{X_I(x,y)\neq Y_I(x,y)\}\leq \sum_{I\in\mathcal{D}}e^{-\tfrac{1}{2\nu}\tfrac{A^2}{\abs{\psi_I(x)}^2\abs{\psi_I(y)}^2}}.
	\end{equation*}
	The series in (ii) of Kolmogorov's Theorem converges since 
	\begin{align*}
	\sum_{I\in\mathcal{D}}\mathscr{E}\abs{Y_I(x,y)}&\leq \sum_{I\in\mathcal{D}}\mathscr{E}\abs{X_I(x,y)}\\
	&=\sum_{I\in\mathcal{D}}\mathscr{E}\abs{a_I-\mathscr{E}a_I}\abs{\psi_I(x)}\abs{\psi_I(y)}\\
	&\leq 2\sum_{I\in\mathcal{D}}\mathscr{E}\abs{a_I}\abs{\psi_I(x)}\abs{\psi_I(y)}\\&\leq M\frac{1}{\abs{x-y}},
	\end{align*}
	for some $M>0$.
	Let us finally check the convergence of the third series of Kolmogorov. In fact
	\begin{align*}
	\sigma^2(Y_I(x,y)) &= \mathscr{E}(Y_I(x,y))^2-(\mathscr{E}Y_I(x,y))^2\\
	&\leq \mathscr{E}(X_I(x,y))^2 + (\mathscr{E}\abs{X_I(x,y)})^2\\
	&= \mathscr{E}((a_I-\mathscr{E}a_I)\psi_I(x)\psi_I(y))^2 + (\mathscr{E}\abs{X_I(x,y)})^2.
	\end{align*}
	From the estimate for (ii) we see that $\sum_{I\in\mathcal{D}}(\mathscr{E}\abs{X_I(x,y)})^2$ is finite for $x\neq y$. Let us use the fact that $a_I\in\mathscr{G}(\nu)$ to estimate the first term above. Write
	\begin{align*}
	\mathscr{E}((a_I-\mathscr{E}a_I)\psi_I(x)\psi_I(y))^2 &=
	\int_{\Omega}(a_I-\mathscr{E}a_I)^2\psi_I^2(x)\psi_I^2(y) d\mathscr{P}(\omega)\\
	&= \psi_I^2(x)\psi_I^2(y)\int_\Omega\left(\int_0^{(a_I-\mathscr{E}a_I)^2} dt\right) d\mathscr{P}\\
	&= \psi_I^2(x)\psi_I^2(y)\int_0^\infty\mathscr{P}\{\abs{a_I(\omega)-\mathscr{E}a_I}>\sqrt{t}\} dt\\
	&\leq \psi_I^2(x)\psi_I^2(y)\int_0^\infty e^{-\tfrac{t}{2\nu}}dt \\
	&= 2\nu\psi_I^2(x)\psi_I^2(y).
	\end{align*}
	So that
	\begin{equation*}
	\sum_{I\in\mathcal{D}}\mathscr{E}((a_I-\mathscr{E}a_I)\psi_I(x)\psi_I(y))^2\leq 2\nu	\sum_{I\in\mathcal{D}}\abs{\psi_I(x)}^2\abs{\psi_I(y)}^2\leq 2\nu\frac{c^2}{\abs{x-y}^2},
	\end{equation*}
	and we are done.
\end{proof}
Let us point out that the above result does not involve any assumption of smoothness on the wavelet $\psi$. Hence the result holds also for the Haar wavelet, since being $h_0^0(x)=\mathcal{X}_{[0,1/2)}(x)- \mathcal{X}_{[1/2.1)}(x)$ compactly supported, certainly satisfies the estimate $\abs{h_0^0(x)}\leq C(1+\abs{x})^{-1-\varepsilon}$. 

\section{$K(x,y;\omega)$ as on $L^2(\Omega,d\mathscr{P})$ valued Calder\'on-Zygmund kernel. The case of $\psi$ smooth}\label{sec:KvaluedCZPsiSmooth}

For the main result of this section the wavelet function $\psi$ is assumed to satisfy the classical condition
\begin{equation}\label{eq:waveletfunctionclassicalcondition}
\abs{\psi(x)}+\abs{\psi'(x)}\leq\frac{C}{(1+\abs{x})^{1+\varepsilon}}
\end{equation}
for every $x\in\mathbb{R}$ and some positive constants $C$ and $\varepsilon$. We shall also assume that $\{a_I(\omega): I\in\mathcal{D}\}$ is a sequence of random variables in $(\Omega,\mathscr{F},\mathscr{P})$ such that
\begin{itemize}
	\item[(4.2.$i$)] the $a_I$'s are independent random variables;
	\item[(4.2.$e$)] $\sum_{I\in\mathcal{D}}\mathscr{E}\abs{a_I}<\infty$;
	\item[(4.2.$\sigma$)] $\{a_I: I\in\mathcal{D}\}\subset\mathscr{G}(\nu)$ for some $\nu>0$.
\end{itemize}
Let us start by the $L^2(\mathbb{R},dx)$ theory. Notice that in general the operator
\begin{equation*}
T: f\longrightarrow \sum_{I\in\mathcal{D}} a_I(\omega) \proin{f}{\psi_I} \psi_I(x),
\end{equation*}
for a given $\omega\in\Omega$, is not bounded on $L^2(\mathbb{R})$, since $a_I(\omega)$ can be unbounded as a sequence on $\mathcal{D}$. Nevertheless we have the following result.
\begin{theorem}\label{thm:TboundedL2}
	Assume that the sequence $\{a_I: I\in\mathcal{D}\}$ satisfies $(4.2.i)$, $(4.2.e)$ and $(4.2.\sigma)$. Assume also that $\abs{\psi(x)}\leq\tfrac{C}{(1+\abs{x})^{1+\varepsilon}}$ for some $C>0$, some $\varepsilon>0$ and every $x\in\mathbb{R}$. Then $T$ is bounded as an operator from $L^2(\mathbb{R},dx)$ to $L^2(L^2(\Omega,d\mathscr{P});dx)$.
\end{theorem}
\begin{proof}
	Let us denote with $\normiii{\cdot}_2$ the norm in $L^2(L^2(\Omega,d\mathscr{P});dx)$ and $\norm{\cdot}_2$ the $L^2(dx)$ norm. Then
	\begin{equation*}
	\normiii{Tf}_2^2 = \int_{\mathbb{R}}\left(\int_{\Omega}\abs{\sum_{I\in\mathcal{D}}a_I(\omega)\proin{f}{\psi_I}\psi_I(x)}^2 d\mathscr{P}(\omega)\right) dx.
	\end{equation*}
	For fixed $x\in\mathbb{R}$ we can estimate $\int_{\Omega}\abs{\sum_{I\in\mathcal{D}}a_I(\omega)\proin{f}{\psi_I}\psi_I(x)}^2 d\mathscr{P}(\omega)$ using the fact that the random variables $a_I$ are subgaussian. In the perspective of Proposition~\ref{propo:ConvergenceSerieRVwithfiniteserieexpectation} in Section~\ref{sec:CramerChernoffMethods} above, set $X_I(\omega)=a_I(\omega)\proin{f}{\psi_I}\psi_I(x)$, for $I\in\mathcal{D}$. Since $a_I\in\mathscr{G}(\nu)$, we have that
	\begin{equation*}
	\eta_{a_I-\mathscr{E}a_I}(\lambda)\leq \lambda^2\frac{\nu}{2}.
	\end{equation*}
	Hence
	\begin{align*}
	\eta_{X_I-\mathscr{E}X_I}(\lambda)&=\log\mathscr{E}e^{\lambda\proin{f}{\psi_I}\psi_I(x)(a_I-\mathscr{E}a_I)}\\
	&= \eta_{a_I-\mathscr{E}a_I}(\lambda\abs{\proin{f}{\psi_I}}\abs{\psi_I(x)})\\
	&\leq \frac{\nu}{2}\lambda^2\abs{\proin{f}{\psi_I}}^2\abs{\psi_I(x)}^2.
	\end{align*}
	So that $\{X_I: I\in\mathcal{D}\}$ is a sequence of independent random variables with
	\begin{equation*}
	\sum_{I\in\mathcal{D}}\mathscr{E}\abs{X_I} = \abs{\proin{f}{\psi_I}}\abs{\psi_I(x)}\sum_{I\in\mathcal{D}}\mathscr{E}\abs{a_I}<\infty.
	\end{equation*}
	Also, from the above estimate for $\eta_{X_I-\mathscr{E}X_I}(\lambda)$, we see that $X_I\in\mathscr{G}(\nu\abs{\proin{f}{\psi_I}}^2\abs{\psi_I(x)}^2)$. Since $\int_{\mathbb{R}}\sum_{I\in\mathcal{D}} \nu\abs{\proin{f}{\psi_I}}^2\abs{\psi_I(x)}^2 dx= \nu\sum_{I\in\mathcal{D}}\abs{\proin{f}{\psi_I}}^2\norm{\psi_I}^2=\nu\norm{f}^2$, we have, except for a null set in $\mathbb{R}$, that the series $\sum_{I\in\mathcal{D}}\nu\abs{\proin{f}{\psi_I}}^2\abs{\psi_I(x)}^2=\nu\sum_{I\in\mathcal{D}}\abs{\proin{f}{\psi_I}}^2\abs{\psi_I(x)}^2$ converges. Then from Proposition~\ref{propo:ConvergenceSerieRVwithfiniteserieexpectation} we get that $\sum_{I\in\mathcal{D}}X_I(\omega)$ converges in $L^2(\Omega,d\mathscr{P})$ to a sum $S$ that belongs to $\mathscr{G}\left( 8\nu\sum_{I\in\mathcal{D}}\abs{\proin{f}{\psi_I}}^2\abs{\psi_I(x)}^2 \right)$. Briefly,
	\begin{equation*}
	\eta_{\sum_{I\in\mathcal{D}}(a_I-\mathscr{E}a_I)\proin{f}{\psi_I}\psi_I(x)}(\lambda) \leq e^{-\tfrac{8\nu}{2}\left(\sum_{I\in\mathcal{D}}\abs{\proin{f}{\psi_I}}^2\abs{\psi_I(x)}^2\right)\lambda^2}.
	\end{equation*}
	So that, from Proposition~\ref{propo:inequalitiesProbalityandExponential},
	\begin{equation}\label{eq:inequalityExponentialProbabilitySeriesDiferenceMean}
	\mathscr{P}\left\{\abs{\sum_{I\in\mathcal{D}}(a_I-\mathscr{E}a_I)\proin{f}{\psi_I}\psi_I(x)} >t \right\} \leq 2 e^{-\tfrac{t^2}{4\nu \sum_{I\in\mathcal{D}}\abs{\proin{f}{\psi_I}}^2\abs{\psi_I(x)}^2}}.
	\end{equation}
	With this last estimate in mind we are in position to obtain an upper bound for $\normiii{Tf}^2_2$. In fact, notice first that for $x$ fixed as above,
	\begin{align*}
	\int_{\Omega}\abs{\sum_{I\in\mathcal{D}}(a_I-\mathscr{E}a_I)\proin{f}{\psi_I}\psi_I(x)}^2 d\mathscr{P(\omega)} &= \int_{\Omega}\left(\int_0^{\abs{\sum_{I\in\mathcal{D}}(a_I-\mathscr{E}a_I)\proin{f}{\psi_I}\psi_I(x)}^2 } dt \right) d\mathscr{P}(\omega)\\
	&= \int_0^\infty \left(\int_{\{\omega\in\Omega: \abs{\sum_{I\in\mathcal{D}}(a_I-\mathscr{E}a_I)\proin{f}{\psi_I}\psi_I(x)}>\sqrt{t} \}} d\mathscr{P}(\omega)\right) dt\\
	&\leq 2 \int_0^\infty e^{-\tfrac{t}{4\nu\sum_{I\in\mathcal{D}}\abs{\proin{f}{\psi_I}}^2\abs{\psi_I(x)}^2}} dt\\
	&= 8\nu\sum_{I\in\mathcal{D}}\abs{\proin{f}{\psi_I}}^2\abs{\psi_I(x)}^2.
	\end{align*}
	And
	\begin{align*}
	\normiii{Tf- \Bigl(\sum_{I\in\mathcal{D}}\mathscr{E}a_I\Bigr) f}^2_2 &= \int_{\mathbb{R}}\left(\int_{\Omega}\abs{\sum_{I\in\mathcal{D}}(a_I(\omega)-\mathscr{E}a_I)\proin{f}{\psi_I}\psi_I(x)}^2 d\mathscr{P}(\omega)\right) dx\\
	&\leq 8\nu\sum_{I\in\mathcal{D}}\abs{\proin{f}{\psi_I}}^2\int_{\mathbb{R}} \abs{\psi_I(x)}^2 dx\\
	&= 8\nu\norm{f}_2^2.
	\end{align*}
	Hence
	\begin{equation*}
	\normiii{Tf}_2 \leq  \norm{Tf - \Bigl(\sum_{I\in\mathcal{D}}\mathscr{E}a_I\Bigr) f}_2 + \Bigl(\sum_{I\in\mathcal{D}}\mathscr{E}a_I\Bigr)\norm{f}_2
	\leq \Bigl(\sqrt{8\nu} + \sum_{I\in\mathcal{D}}\mathscr{E}\abs{a_I}\Bigr)\norm{f}_2.
	\end{equation*}
\end{proof}
The kernel of the operator $T$ is given by $K(x,y;\omega)=\sum_{I\in\mathcal{D}}a_I(\omega)\psi_I(x)\psi_I(y)$. We shall think $K(x,y;\cdot)$ as a kernel defined in $\mathbb{R}^2$ with values in $L^2(\Omega,d\mathscr{P})$. The next result contain the basic estimates showing that $K(x,y;\cdot)$ is an $L^2(\Omega,d\mathscr{P})$ valued Calder\'on-Zygmund kernel.

\begin{theorem}\label{thm:controlKandPartialK}
Assume that the wavelet $\psi$ satisfies \eqref{eq:waveletfunctionclassicalcondition} and that the $a_I$'s satisfy (4.2.i), (4.2.e) and (4.2.$\sigma$). Then there exists a constant $B$ such that
\begin{enumerate}[(\ref{thm:controlKandPartialK}.a)]
	\item $\norm{K(x,y,\cdot)}_{L^2(\Omega,d\mathscr{P})}\leq\frac{B}{\abs{x-y}}$, $x, y\in\mathbb{R}$;
	\item $\norm{\frac{\partial K}{\partial x}(x,y,\cdot)}_{L^2(\Omega,d\mathscr{P})} + \norm{\frac{\partial K}{\partial y}(x,y,\cdot)}_{L^2(\Omega,d\mathscr{P})}\leq\frac{B}{\abs{x-y}^2}$, $x, y\in\mathbb{R}$.
\end{enumerate}
\end{theorem}
\begin{proof}
	Since $K(x,y;\omega)=\Sigma(x,y;\omega)+\sum_{I\in\mathcal{D}}\mathscr{E}a_I\psi_I(x)\psi_I(y)$ with $\Sigma(x,y;\omega)=\sum_{I\in\mathcal{D}}(a_I(\omega)-\mathscr{E}a_I)\psi_I(x)\psi_I(y)$, and from the classical result in \cite{Daubechies92} we have that $\abs{\sum_{I\in\mathcal{D}}\mathscr{E}a_I\psi_I(x)\psi_I(y)}$ and its partial derivatives satisfy the desired estimates, it is enough to prove (4.2.a) and (4.2.b) with $\Sigma(x,y;\omega)$ instead of $K(x,y;\omega)$. Let us start proving (4.2.a) for $\Sigma(x,y,\cdot)$. Let us use again Proposition~\ref{propo:ConvergenceSerieRVwithfiniteserieexpectation}. Take now $X_I(\omega)=(a_I(\omega)-\mathscr{E}a_I)\psi_I(x)\psi_I(y)$ for $x\neq y$ both fixed. Notice first that
	\begin{align*}
	\sum_{I\in\mathcal{D}}\mathscr{E}\abs{X_I} &= \sum_{I\in\mathcal{D}}\mathscr{E}(\abs{a_I(\omega)-\mathscr{E}a_I})\abs{\psi_I(x)}\abs{\psi_I(y)}\\
	&\leq 2 \sum_{I\in\mathcal{D}}\mathscr{E}\abs{a_I}\abs{\psi_I(x)}\abs{\psi_I(y)}\\
	&\leq 2\sup_{I\in\mathcal{D}}\mathscr{E}\abs{a_I}\frac{c}{\abs{x-y}}<\infty.
	\end{align*}
	Also
	\begin{align*}
	\eta_{X_I}(\lambda) &= \eta_{(a_I(\omega)-\mathscr{E}a_I)\psi_I(x)\psi_I(y)}(\lambda)\\
	&=\log \mathscr{E}e^{\lambda\psi_I(x)\psi_I(y)(a_I(\omega)-\mathscr{E}a_I)}\\
	&= \eta_{a_I-\mathscr{E}a_I}(\lambda\abs{\psi_I(x)}\abs{\psi_I(y)})\\
	&\leq\frac{\nu}{2}\lambda^2\abs{\psi_I(x)}^2\abs{\psi_I(y)}^2
	\end{align*}
	so that $\{X_I: I\in\mathcal{D}\}$ is a sequence of independent random variables with $\sum_{I\in\mathcal{D}}\mathscr{E}\abs{X_I}<\infty$ and $X_I\in\mathscr{G}(\nu\abs{\psi_I(x)}^2\abs{\psi_I(y)}^2)$.
	On the other hand, the estimates in \cite{Daubechies92} show that the series $\sum_{I\in\mathcal{D}}\nu\abs{\psi_I(x)}^2\abs{\psi_I(y)}^2=\nu\sum_{I\in\mathcal{D}}\abs{\psi_I(x)}^2\abs{\psi_I(y)}^2$ converges. Then, from Proposition~\ref{propo:ConvergenceSerieRVwithfiniteserieexpectation}, we have that $\Sigma(x,y;\omega)\in\mathscr{G}(8\nu\sum_{I\in\mathcal{D}}\abs{\psi_I(x)}^2\abs{\psi_I(y)}^2)$. Now, from Proposition~\ref{propo:inequalitiesProbalityandExponential}, we get
	\begin{equation}\label{eq:inequalityExponentialProbabilitySeriesSigma}
	\mathscr{P}\{\abs{\Sigma(x,y;\omega)}>t\}\leq 2e^{-\frac{t^2}{4\nu\sum_{I\in\mathcal{D}}\abs{\psi_I(x)}^2\abs{\psi_I(y)}^2}}.
	\end{equation}
	Hence
	\begin{align*}
	\norm{\Sigma(x,y;\cdot)}^2_{L^2(\Omega,d\mathscr{P})} &= \int_{\Omega}\abs{\Sigma(x,y;\omega)}^2 d\mathscr{P}\\
	&= \int_{\Omega}\left(\int_0^{\abs{\Sigma(x,y;\omega)}^2} dt\right) d\mathscr{P}\\
	&= \int_0^\infty\mathscr{P}\{\abs{\Sigma(x,y;\omega)}^2>t\} dt\\
	&\leq 2 \int_0^\infty e^{-\frac{t}{4\nu\sum_{I\in\mathcal{D}}\abs{\psi_I(x)}^2\abs{\psi_I(y)}^2}} dt\\
	&=8\nu \sum_{I\in\mathcal{D}}\abs{\psi_I(x)}^2\abs{\psi_I(y)}^2\\
	&\leq \frac{8\nu C}{\abs{x-y}^2},
	\end{align*}
	and (4.2.a) is proved for $\Sigma$.
	
	Let us now prove (4.2.b). It suffices to show that $\norm{\frac{\partial \Sigma}{\partial x}(x,y,\cdot)}_{L^2(\Omega,d\mathscr{P})} \leq\frac{B}{\abs{x-y}^2}$. With the arguments in the proof of Theorem~3.1 and the assumptions on $\psi$ and $\psi'$, we have that the series $\sum_{I\in\mathcal{D}}(a_I(\omega)-\mathscr{E}a_I)\abs{I}^{-2}\psi'(2^{j(I)}x-k(I)) \psi(2^{j(I)}y-k(I))$ converges for almost every $\omega\in\Omega$. Then 
	\begin{equation*}
	\frac{\partial\Sigma}{\partial x}(x,y;\omega) =\sum_{I\in\mathcal{D}}(a_I(\omega)-\mathscr{E}a_I)\abs{I}^{-1}\widetilde{\psi}_I(x)\psi_I(y),
	\end{equation*}
	where $\widetilde{\psi}=\frac{d\psi}{dx}$. Since $\psi$ and $\widetilde{\psi}$ have the same size estimate, we can proceed as in the proof of (4.2.a). In fact, we shall use again Proposition~\ref{propo:ConvergenceSerieRVwithfiniteserieexpectation} with $X_I(\omega)=(a_I(\omega)-\mathscr{E}a_I)\abs{I}^{-1}\widetilde{\psi}_I(x)\psi_I(y)$, for $x\neq y$. Now
	\begin{align*}
	\sum_{I\in\mathcal{D}}\mathscr{E}\abs{X_I} &\leq 2\sum_{I\in\mathcal{D}}\mathscr{E}\abs{a_I}\abs{I}^{-1}\widetilde{\psi}_I(x)\psi_I(y)\\
	&\leq 2\sup_{I\in\mathcal{D}}\mathscr{E}\abs{a_I}\left(\sum_{I\in\mathcal{D}}\abs{I}^{-1}\widetilde{\psi}_I(x)\psi_I(y)\right)\\
	&\leq 2\sup_{I\in\mathcal{D}}\mathscr{E}\abs{a_I}\frac{C}{\abs{x-y}^2}.
	\end{align*}
	Also
	\begin{align*}
	\eta_{X_I}(\lambda) &= \log\mathscr{E} e^{\lambda\abs{I}^{-1}\widetilde{\psi}_I(x)\psi_I(y)(a_I-\mathscr{E}a_I)}\\
	&= \eta_{a_I-\mathscr{E}a_I} (\lambda\abs{I}^{-1}\abs{\widetilde{\psi}_I(x)}\abs{\psi_I(y)})\\
	&\leq \frac{\nu}{2}\lambda^2\abs{I}^{-2}\abs{\widetilde{\psi}_I(x)}^2\abs{\psi_I(y)}^2.
	\end{align*}
	Hence, from Proposition~\ref{propo:ConvergenceSerieRVwithfiniteserieexpectation},
	\begin{equation*}
	\frac{\partial\Sigma}{\partial x}(x,y;\omega)\in\mathscr{G}\left(8\nu \sum_{I\in\mathcal{D}}\abs{I}^{-2}\abs{\widetilde{\psi}_I(x)}^2\abs{\psi_I(y)}^2\right).
	\end{equation*}
	Then
	\begin{align*}
	\norm{\frac{\partial\Sigma}{\partial x}(x,y;\cdot)}^2_{L^2(\Omega,d\mathscr{P})} &= \int_{\Omega} \abs{\frac{\partial\Sigma}{\partial x}(x,y;\omega)}^2 d\mathscr{P}\\
	&\leq 2 \int_0^\infty e^{-\frac{t}{4\nu\sum_{I\in\mathcal{D}}\abs{I}^{-2}\abs{\widetilde{\psi}_I(x)}^2\abs{\psi_I(y)}^2}} dt\\
	&= 8\nu\sum_{I\in\mathcal{D}}\abs{I}^{-2}\abs{\widetilde{\psi}_I(x)}^2\abs{\psi_I(y)}^2\\
	&\leq \frac{8\nu C}{\abs{x-y}^4},
	\end{align*}
	the last estimate follows again as in \cite{Daubechies92}.
\end{proof}

Now the boundedness properties of $T$ follow from the general results on vector valued singular integrals in \cite{RdFRuTo86} or \cite{GraLiDa09}.
\begin{theorem}\label{thm:TfisBoundedOperatorLp}
	Assume that the wavelet $\psi$ satisfies \eqref{eq:waveletfunctionclassicalcondition} and that the $a_I$'s satisfy (4.2.i), (4.2.e) and (4.2.$\sigma$). Then for $1<p<\infty$, $Tf=\sum_{I\in\mathcal{D}}a_I\proin{f}{\psi_I}\psi_I$ is bounded as an operator from $L^p(\mathbb{R},dx)$ to $L^p(L^2(\Omega,d\mathscr{P});dx)$. Moreover,
	\begin{equation*}
	\abs{\{x\in\mathbb{R}: \norm{Tf (x)}_{L^2(\Omega,d\mathscr{P})}>t\}}\leq \frac{C}{\lambda}\norm{f}_{L^1(\mathbb{R},dx)}.
	\end{equation*}
\end{theorem}

\section{$K(x,y;\omega)$ as an $L^2(\Omega,d\mathscr{P})$ valued Calder\'on-Zygmund kernel defined in the space of homogeneous type $(\mathbb{R}^+, \delta, \abs{\cdot})$. The case of the Haar wavelet}\label{sec:KvaluedCZHaarwavelet}

Let us observe first that since the function $\psi(x)=\mathcal{X}_{[0,1/2)}(x)-\mathcal{X}_{[1/2,1)}(x)$ satisfies the basic size estimate $\abs{\psi(x)}\leq\frac{C}{(1+\abs{x})^{1+\varepsilon}}$, all the results in the previous section which do not involve smoothness holds for the Haar wavelet. In this section we shall briefly sketch the results for the Haar wavelet following the lines in \cite{AiGoPetermichl18} where a natural metric structure in $\mathbb{R}^+$ allows to use the general theory of Calder\'on-Zygmund Singular Integrals.

In particular Theorem~\ref{thm:ConvergenceSeriesWaveletpsi} and Theorem~\ref{thm:TboundedL2} hold for the Haar function. The only results that needs to be considered is an analogous of Theorem~\ref{thm:controlKandPartialK}.

Set $\mathbb{R}^+$ to denote the set of nonnegative real numbers and $\mathcal{D}^+$ the set of dyadic intervals in $\mathbb{R}^+$. The set $\mathbb{R}^+$ with Lebesgue measure and the dyadic distance $\delta(x,y)=\inf\{\abs{I}: x, y\in I, I\in\mathcal{D}^+\}$ is a space of homogeneous type. Actually $(\mathbb{R}^+, \delta, \abs{\cdot})$ is a $1$-Ahlfors regular or normal space. Moreover, the kernel $K(x,y,\cdot)$ valued in $L^2(\Omega, d\mathscr{P})$ is a Calder\'on-Zygmund kernel in this space of homogeneous type.
\begin{theorem}\label{thm:kernelKisCalderonZygmund}
	Let $\psi$ be the Haar wavelet. Assume that the $a_I$'s satisfy (4.2.i), (4.2.e) and (4.2.$\sigma$). Then there exists a constant $B$ such that
	\begin{itemize}
	\item[(\ref{thm:kernelKisCalderonZygmund}.a)] $\norm{K(x,y,\cdot)}_{L^2(\Omega,d\mathscr{P})}\leq \frac{B}{\delta(x,y)}$, $x, y\in\mathbb{R}^+$;
	\item[(\ref{thm:kernelKisCalderonZygmund}.b.i)]
	$\norm{K(x',y;\cdot)-K(x,y;\cdot)}_{L^2(\Omega,d\mathscr{P})}\leq B\frac{\delta(x',x)}{\delta(x,y)^2}$, when $2\delta(x',x)\leq\delta(x,y)$;
	\item[(\ref{thm:kernelKisCalderonZygmund}.b.ii)]
	$\norm{K(x,y';\cdot)-K(x,y;\cdot)}_{L^2(\Omega,d\mathscr{P})}\leq B\frac{\delta(y,y')}{\delta(x,y)^2}$, when $2\delta(y',y)\leq\delta(x,y)$.
	\end{itemize}
\end{theorem}
Let us point out here that once the above results is proved, the analogous of Theorem~\ref{thm:TfisBoundedOperatorLp} for the Haar system follow from the general setting of the Calder\'on-Zygmund theory given in \cite{GraLiDa09}.
\begin{proof}[Proof of Theorem~\ref{thm:kernelKisCalderonZygmund}]
Let us start with (\ref{thm:kernelKisCalderonZygmund}.a). Notice that (4.1.a) holds since only th size condition on $\psi$ is used in its proof. Nevertheless, since $\delta(x,y)\geq \abs{x-y}$ but there metrics are not equivalent, (\ref{thm:kernelKisCalderonZygmund}.a) is a better estimate for the size of $K$ which can not be directly obtained from (4.2.a). For $x\neq y$ both in $\mathbb{R}^+$, with the notation in the proof of Theorem~\ref{thm:controlKandPartialK}, we have that inequality \eqref{eq:inequalityExponentialProbabilitySeriesSigma} holds \textit{mutatis mutandis} for $\psi$ the Haar wavelet. Then
\begin{equation*}
\norm{\Sigma(x,y;\cdot)}^2_{L^2(\Omega,d\mathscr{P})}\leq 8\nu\sum_{I\in\mathcal{D}} \abs{\psi_I(x)}^2\abs{\psi_I(y)}^2.
\end{equation*}
Let us estimate the series in the right hand side of the above inequality. Let $I(x,y)$ be the smallest dyadic interval in $\mathcal{D}^+$ containing both, $x$ and $y$. Set $I^l$ to denote the $l^{th}$ ancestor of $I(x,y)$. Precisely $I^0=I(x,y)$, $I^1$ the only interval in $\mathcal{D}^+$ containing $I^0$ with $\abs{I^1}=2\abs{I^0}$. For $l=2$, $I^2\subset I^1$, $I^2\in\mathcal{D}^+$ and $\abs{I^2}=2\abs{I^1}$ and so on. Notice that for each $I\subsetneq I^0$ we have that $\abs{\psi_I(x)}^2\abs{\psi_I(y)}^2=0$ since $x$ or $y$ does not belong to $I$, being $I^0$ the smallest interval in $\mathcal{D}^+$ containing $x$ and $y$. Hence
\begin{align*}
\sum_{I\in\mathcal{D}} \abs{\psi_I(x)}^2\abs{\psi_I(y)}^2 &= \sum_{l\geq 0}\abs{\psi_{I^l}(x)}^2\abs{\psi_{I^l}(y)}^2\\
&=\sum_{l\geq 0}\abs{I^l}^{-2}\\
&= \sum_{l\geq 0} (2^l\abs{I^0})^{-2} \\
&=\abs{I^0}^{-2}\sum_{l\geq 0} 4^{-l}\\
&=\frac{4}{3}\frac{1}{\abs{I(x,y)}^2}\\
&=\frac{4}{3}\frac{1}{\delta(x,y)^2},
\end{align*}
and (\ref{thm:kernelKisCalderonZygmund}.a) is proved.

Let us prove (\ref{thm:kernelKisCalderonZygmund}.b.i). The second estimate can be handled in a similar way. 
With the above notation, for fixed $\omega\in\Omega$ we have that
\begin{align*}
K(x',y;\omega)-K(x,y;\omega) &=\sum_{I\in\mathcal{D}}a_I(\omega) (\psi_I(x')-\psi_I(x)) \psi_I(y)\\
&=\sum_{l\geq 0}a_{I^l}(\omega)(\psi_{I^l}(x')-\psi_{I^l}(x)) \psi_{I^l}(y).
\end{align*}
Now, since $\abs{I^0}=\abs{I(x,y)}=\delta(x,y)\geq 2\delta(x,x')$, $x$ and $x'$ must belong to the same half of $I(x,y)$. And hence $x$ and $x'$ belong to the same half of each $I^l$. Then $\psi_{I^l}(x)=\psi_{I^l}(x')$ and $K(x',y;\omega)=K(x,y;\omega)$. And we are done.
\end{proof}

\section{Concentration}\label{sec:Concentration}

In all the results of the previous sections we have been dealing with a sequence of random variables $\{a_I: I\in\mathcal{D}\}$ satisfying (4.2.i), (4.2.e) and (4.2.$\sigma$). In particular the kernels $K(x,y;\omega)$ and the induced operators $T_\omega$, have mean values given by 
\begin{equation*}
K(x,y)=\sum_{I\in\mathcal{D}}\mathscr{E}a_I \psi_I(x)\psi_I(y)
\end{equation*}
and 
\begin{equation*}
Tf(x) = \sum_{I\in\mathcal{D}}\mathscr{E}a_I\proin{f}{\psi_I}\psi_I(x).
\end{equation*}
Since from (4.2.e) the sequence $\{\mathscr{E}a_I: I\in\mathcal{D}\}$ is bounded, $K$ is a Calder\'on-Zygmund kernel and $T$ a Calder\'on-Zygmund operator both scalar valued. Inequalities \eqref{eq:inequalityExponentialProbabilitySeriesSigma} and \eqref{eq:inequalityExponentialProbabilitySeriesDiferenceMean} give estimates for the concentration of $K(x,y;\omega)$ about $K(x,y)$ and of $T_\omega$ about $T$. In particular the subgaussian character of the distribution
\begin{equation*}
\mathscr{P}\{\omega: \abs{K(x,y;\omega)-K(x,y)}>t\}
\end{equation*}
reveals as a variance factor the reciprocal of the underlying metric in the space.

Even when the main steps have already been proved in Theorems~\ref{thm:TboundedL2},  \ref{thm:controlKandPartialK} and  \ref{thm:kernelKisCalderonZygmund}, let us state these estimates.

\begin{theorem}
\begin{enumerate}[(A)]
\item Let $\{a_I: I\in\mathcal{D}\}$ be a sequence of random variables satisfying (4.2.i), (4.2.e) and (4.2.$\sigma$). Let $\psi$ be a wavelet function satisfying $\abs{\psi(x)}\leq C(1+\abs{x})^{-1-\varepsilon}$. Then, for every $t>0$,
		\begin{equation*}
		\mathscr{P}\{\abs{K(x,y;\omega)-K(x,y)}>t\} \leq 2 e^{-\tfrac{C^2}{4\nu}\abs{x-y}^2 t^2}.
		\end{equation*}
\item Let $\{a_I: I\in\mathcal{D}^+\}$ be a sequence of random variables satisfying (4.2.i), (4.2.e) and (4.2.$\sigma$). Let $\psi(x)=\mathcal{X}_{[0,1/2)}(x)-\mathcal{X}_{[1/2,1)}(x)$ be the Haar wavelet. Then, for every $t>0$,
		\begin{equation*}
		\mathscr{P}\{\abs{K(x,y;\omega)-K(x,y)}>t\} \leq 2 e^{-\tfrac{C^2}{4\nu}\delta(x,y)^2 t^2}.
		\end{equation*}
\item For $\{a_I\}$ as before and $\psi$ satisfying \eqref{eq:waveletfunctionclassicalcondition} or with $\psi$ the Haar wavelet, we have
		\begin{equation*}
		\mathscr{P}\{\abs{T_\omega f(x)- Tf(x)}>t\} \leq 2 e^{-\tfrac{t^2}{4\nu \sum_{I\in\mathcal{D}}\abs{\proin{f}{\psi_I}}^2\abs{\psi_I(x)}^2}}.
		\end{equation*}
\end{enumerate}
\end{theorem}
\begin{proof}
(A) follows from \eqref{eq:inequalityExponentialProbabilitySeriesSigma} and the standard estimates in \cite{Daubechies92}. (B) follows from \eqref{eq:inequalityExponentialProbabilitySeriesSigma} and the estimate in the proof of Theorem~\ref{thm:kernelKisCalderonZygmund}. (C) follows from \eqref{eq:waveletfunctionclassicalcondition}. 
\end{proof}

Let us finally observe that in the case of Rademacher random variables the above concentration estimates hold with $\nu=1$, $K(x,y)\equiv 0$ and $T\equiv 0$.

%%%%%%%%%%%%%  References %%%%%%%%%%%%%%%%%
%\bibliographystyle{amsalpha}
%\bibliography{ref}

\providecommand{\bysame}{\leavevmode\hbox to3em{\hrulefill}\thinspace}
\providecommand{\MR}{\relax\ifhmode\unskip\space\fi MR }
\providecommand{\MRhref}[2]{%
	\href{http://www.ams.org/mathscinet-getitem?mr=#1}{#2}
}
\providecommand{\href}[2]{#2}

%%%%%%%%%%%%% Address Authors %%%%%%%%%%%%%%%

\bigskip
\noindent{\footnotesize
\noindent\textit{Affiliations:\,}
\textsc{Instituto de Matem\'{a}tica Aplicada del Litoral, UNL, CONICET.}

\smallskip
\noindent\textit{Address:\,}\textmd{CCT CONICET Santa Fe, Predio ``Alberto Cassano'', Colectora Ruta Nac.~168 km 0, Paraje El Pozo, S3007ABA Santa Fe, Argentina.}

\smallskip
\noindent \textit{E-mail address:\, }\verb|haimar@santafe-conicet.gov.ar|; \verb|ivanagomez@santafe-conicet.gov.ar|

}

\end{document}